\documentclass{amsart}   	
\usepackage{lineno,hyperref}
\modulolinenumbers[5]

\usepackage{geometry}                		
\usepackage{amsmath}
\usepackage{enumerate}
\usepackage{amsthm}	
\usepackage{amssymb}
\usepackage{bbold}
\usepackage{tikz}
\usepackage[toc]{appendix}

\newtheorem{thm}{Theorem}[section]

\newtheorem{prop}[thm]{Proposition}

\newtheorem{lem}[thm]{Lemma}

\newtheorem*{claim}{Claim}

\newtheorem*{thm*}{Theorem}
\newtheorem*{question*}{Question}
\newtheorem*{assumption*}{Assumption}

\newcommand{\F}{\mathbb{F}}
\newcommand{\fl}{\mathfrak{l}}

\newcommand{\Gal}{\mathrm{Gal}}
\newcommand{\Aut}{\mathrm{Aut}}
\newcommand{\sep}{\mathrm{sep}}
\newcommand{\alg}{\mathrm{alg}}
\newcommand{\GL}{\mathrm{GL}}
\newcommand{\SL}{\mathrm{SL}}

\newenvironment{usethmcounterof}[1]{%
  \renewcommand\thethm{\ref{#1}}\thm}{\endthm\addtocounter{thm}{-1}}

\title{Exceptional cases of adelic surjectivity for Drinfeld modules of rank 2}
\author{Chien-Hua Chen}

\begin{document}
\begin{abstract}
In this paper, we study the surjectivity of adelic Galois representation associated to Drinfeld $\F_q[T]$-modules over $\F_q(T)$ of rank $2$ in the cases when $q$ is even or $q=3$.

\end{abstract}

\maketitle

\section{Introduction}
In \cite{Ser72}, Serre proved the adelic openness for elliptic curves $E$ over $\mathbb{Q}$ without complex multiplication. In other words, the image of the Galois representation
$$\rho_E:\Gal(\bar{\mathbb{Q}}/\mathbb{Q})\rightarrow \varprojlim_{n}{\rm{Aut}}E[n]\cong \GL_2(\hat{\mathbb{Z}})$$
is an open subgroup of $\GL_2(\hat{\mathbb{Z}})$. Hence $\rho_{E}(\Gal(\bar{\mathbb{Q}}/\mathbb{Q}))$ has finite index in $\GL_2(\hat{\mathbb{Z}})$.

He then studied the maximal Galois image for elliptic curves over $\mathbb{Q}$ and proved that it is impossible to attain adelic surjectivity for elliptic curves over $\mathbb{Q}$ (see \cite{Ser72}, Proposition 22). The reason is that the $2$-division field $\mathbb{Q}(E[2])$ has a subextension $M$ of degree $2$ over $\mathbb{Q}$, which is also contained in some $m$-division field $\mathbb{Q}(E[m])$. Hence the $2$-division field has nontrivial intersection with the $m$-division field. 

For the function field analogue of the maximal Galois image problem, Pink and R\"utsche\cite{PR09} proved the open image theorem for Drinfeld modules of arbitrary rank with generic characteristic. We let $A=\F_q[T]$ and $F=\F_q(T)$. It is natural then to ask whether there is a Drinfeld $A$-module $\phi$ over $F$ of rank $r$ with generic characteristic whose adelic Galois representation
$${\rho}_{\phi}:G_K\longrightarrow \varprojlim_{\mathfrak{a}}{\rm{Aut}}(\phi[\mathfrak{a}])\cong {\rm{GL_r}}(\widehat{A})$$
is surjective.  Zywina \cite{Zy11} proved the adelic surjectivity for a Drinfeld $A$-module over $F$ of rank $2$ under the assumption that $q\geqslant 5$ is an odd prime power. The author \cite{Chen20} proved the adelic surjectivity for a Drinfeld $A$-module over $F$ of rank $3$ under the condition that $q=p^e$ is a prime power with $p\geqslant 5 \text{ and } q\equiv 1\mod 3$.

One then curious about what happens in the even characteristic case and in the case $q=3$. Does the entanglement between two division fields that is always present for elliptic curves over $\mathbb{Q}$ is also present in Drinfeld modules setting? In this paper, we prove that when $q=2$, the adelic surjectivity for Drinfeld modules over $F$ of rank $2$ is never achieved.

\begin{usethmcounterof}{q=2}
Let $A=\mathbb{F}_2[T]$ and $F=\mathbb{F}_2(T)$. Let $\phi_T=T+g_1\tau+g_2\tau^2$ be a Drinfeld $A$-module over $F$ of rank $2$ where $g_1\in F$ and $g_2\in F^*$. The adelic Galois representation
$${\rho}_{\phi}:G_F\longrightarrow \varprojlim_{\mathfrak{a}}{\rm{Aut}}(\phi[\mathfrak{a}])\cong {\rm{GL_2}}(\widehat{A})$$
is not surjective.

\end{usethmcounterof}
In the case $q=3$, we prove that a certain type of Drinfeld $A$-module over $F$ of rank $2$ will never produce adelic surjectivity. But we do have Drinfeld $A$-module over $F$ of rank $2$ with surjective adelic Galois representation.

\begin{usethmcounterof}{5.3}
Let $A=\F_3[T]$ and $F=\F_3(T)$. Let $\phi$ be a rank $2$ Drinfeld $A$-module over $F$ defined by 
$$\phi_T=T+g_1\tau+g_2\tau^2,\ \text{where $g_1\in F$ and $g_2\in F^*$ with $-g_2$ is a non-square element}.$$
Then the adelic Galois representation 
$${\rho}_{\phi}:G_F\longrightarrow \varprojlim_{\mathfrak{a}}{\rm{Aut}}(\phi[\mathfrak{a}])\cong {\rm{GL_2}}(\widehat{A})$$
is not surjective.

\end{usethmcounterof}

\begin{usethmcounterof}{q=3-2}
Let $A=\F_3[T]$ and $F=\F_3(T)$. Let $\varphi$ be a rank $2$ Drinfeld $A$-module over $F$ defined by 
$$\varphi_T=T+\tau-T^2\tau^2.$$
Then the adelic Galois representation 
$${\rho}_{\varphi}:G_F\longrightarrow \varprojlim_{\mathfrak{a}}{\rm{Aut}}(\varphi[\mathfrak{a}])\cong {\rm{GL_2}}(\widehat{A})$$
is surjective.
\end{usethmcounterof}

Finally, for the case $q=2^e\geqslant 4$, we prove that Zywina's choice of Drinfeld module defined by $\varphi_T=T+\tau-T^{q-1}\tau^2$ can still have adelic surjectivity.

\begin{usethmcounterof}{6.2}

Let $A=\F_q[T]$ and $F=\F_q(T)$ with $q=2^e\geqslant 4$. Let $\varphi$ be a rank $2$ Drinfeld $A$-module over $F$ defined by 
$$\varphi_T=T+\tau-T^{q-1}\tau^2.$$
Then the adelic Galois representation 
$${\rho}_{\varphi}:G_F\longrightarrow \varprojlim_{\mathfrak{a}}{\rm{Aut}}(\varphi[\mathfrak{a}])\cong {\rm{GL_2}}(\widehat{A})$$
is surjective.

\end{usethmcounterof}

\section{Preliminaries}
Let $\mathbb{F}_q$ be a finite field with $q=p^e$, and $A=\mathbb{F}_q[T]$ (we don't restrict $q$ until we are into a specific section). An \textbf{$A$-field} K is a field equipped with a homomorphism $\gamma: A\rightarrow K$ of $\mathbb{F}_q$-algebras. The kernel $\text{ker}(\gamma)$ is called the \textbf{A-characteristic} of $K$, and we say $K$ has \text{generic characteristic} if $\text{ker}(\gamma)=0$.  Let $K\{\tau\}$ be the ring of skew polynomials along with the commutative rule $\tau\cdot c=c^q\cdot \tau$. A \textbf{Drinfeld $A$-module over $K$ of rank $r\geqslant 1$} is a ring homomorphism
\begin{align*}
\phi&: A\longrightarrow K\{\tau\} \\
      &\ \ \ a \longmapsto \phi_a=\gamma(a)+ \sum^{r\cdot {\rm{deg}}(a)}_{i=1}g_i(a)\tau^i.
\end{align*}
It is uniquely determined by $\phi_T=\gamma(T)+\sum^{r}_{i=1}g_i(T)\tau^i$, where $g_r(T)\neq0$. Two Drinfeld modules $\phi$ and $\psi$ are isomorphic over $K$ if there is an element $c\in K^*$ such that $c\cdot \phi_a=\psi_a\cdot c$ for all elements $a\in A$. On the other hand, let $K\langle x\rangle=\{\sum^{n}_{i=0}a_ix^{q^i}\mid a_i\in K\}$ denote the polynomial ring of $q$-polynomials where the multiplication is defined to be composition of polynomials. It's easy to see that $K\{\tau\}$ is isomorphic to $K\langle x\rangle$ under the map
$$
\begin{array}{ccc}
K\{\tau\}&\longrightarrow& K\langle x\rangle \\
\sum^{n}_{i=0}a_i\tau^i &\mapsto& \sum^{n}_{i=0}a_ix^{q^i}
\end{array}
.$$
Thus we may also view $\phi_T$ as a $q$-polynomial $\phi_T(x)=\gamma(T)x+\sum^{r}_{i=1}g_i(T)x^{q^i}$.
The Drinfeld module $\phi$ over $K$ gives $K$ an $A$-module structure, where $a\in A$ acts on $K$ via $\phi_{a}$. We use the notation $^{\phi}K$ to emphasize the action of $A$ on $K$.
The {\textbf{$a$-torsion}} is 
$$\phi[a]=\left\{ \alpha\in K^\alg \mid \phi_a(\alpha)= \gamma(a)\alpha+ \sum^{r\cdot {\rm{deg}}(a)}_{i=1}g_i(a)\alpha^{q^i}=0 \right\}.$$ 
Note that $\phi[a]$ is an $A$-module, where $b\in A$ acts on $\alpha\in \phi[a]$ by 
$b\cdot \alpha=\phi_b(\alpha)$.

\begin{prop}\label{prop0.2}
Let $\phi$ be a Drinfeld module over $K$ of rank $r$ and $0\neq a\in A$. If the $A$-characteristic of $K$ does not divide $a$, then 
there is an isomorphism of $A$-modules $\phi[a]\simeq (A/aA)^r$.
\end{prop}
 
If the characteristic of $K$ does not divide $a$, then $\phi_a(x)=\gamma(a)x+\sum^{r\cdot {\rm{deg}}(a)}_{i=1}g_i(a)x^{q^i}$ is a separable polynomial. Therefore $G_K=\Gal(K^\sep/K)$ acts on $\phi[a]$ and this action commutes with the action of $A$. From this action we get a 
representation $G_K\to \Aut_A(\phi[a])\cong \GL_r(A/aA)$. When $a=\fl$ is irreducible, we denoted this representation 
$$\bar{\rho}_{\phi, \fl}:G_F\rightarrow \GL_r(\mathbb{F}_\mathfrak{l}).$$ 
Obviously $\bar{\rho}_{\phi,\mathfrak{l}}$ factor through $\Gal(F(\phi[\mathfrak{l}])/F)$, where the field $F(\phi[\mathfrak{l}])$ is defined to be the splitting field of $\phi_\mathfrak{l}(x)$ over $F$.

Taking inverse limit with respect to $\fl^i$, we get the \textbf{$\fl$-adic Galois representation}
$${\rho}_{\phi,\fl}: G_K \longrightarrow \varprojlim_{i}{\Aut}(\phi[\fl^i])\cong \GL_r(A_\fl),$$
where $A_\fl$ denotes the completion of $A$ at $\fl$. 
Combining all representations together, we get the \textbf{adelic Galois representation}
$${\rho}_{\phi}:G_K\longrightarrow \varprojlim_{\mathfrak{a}}{\rm{Aut}}(\phi[\mathfrak{a}])\cong {\rm{GL_r}}(\widehat{A}),$$
where $\hat{A}$ is defined as $\varprojlim_{\mathfrak{a}}{A/\mathfrak{a}}$.

Throughout this paper, we set $K$ to be $F=\mathbb{F}_q(T)$ equipped with the natural injection $\gamma: A\rightarrow F$. We will focus on Dinfeld $A$-modules $\phi$ over $F$ of rank $2$, which are usually defined by $\phi_T=T+g_1\tau+g_2\tau^2$ with $g_1\in F$ and $g_2\in F^*$.

\section{Abelian extensions over rational function field}
In the class field theory, Kronecker-Weber theorem says an abelian extension over $\mathbb{Q}$ must be contained in some cyclotomic extension of $\mathbb{Q}$.  In his work on function field analogue of class field theory, Hayes \cite{Hay74} gave a complete description on the maximal abelian extension of $\mathbb{F}_q(T)$. This section is to summarize Hayes' results. 

\begin{enumerate}

\item Maximal constant extension $E$:\\
The maximal constant extension $E$ of $\mathbb{F}_q(T)$ is the union of all constant extensions of $\\\mathbb{F}_q(T)$.

\item Carlitz extension $K_T$:\\
Let $C_T=T+\tau$ be the rank $1$ Drinfeld $A$-module over $F$, which is the so-called Carlitz module. Let $\mathfrak{a}$ be an ideal of $A$, we denote by $F(C[\mathfrak{a}])$ the splitting field of $C_{\mathfrak{a}}(x)$ over $F$. The Carlitz extension $K_T$ is defined to be the union of $F(C[\mathfrak{a}])$ among all ideals in $A$.

\item Wildly ramified at infinity $L_\infty$:\\
Let $A'$ be the polynomial ring $\mathbb{F}_q[\frac{1}{T}]$, consider the Drinfeld $A'$-module $C'_{\frac{1}{T}}=\frac{1}{T}+\tau$ over $F$. Set $F_\nu=F(C'[(\frac{1}{T^{-\nu-1}})])$, we have $[F_\nu: F]=q^\nu(q-1)$. The action of $b\in A'$ on $F_\nu$ is defined via $C'_b$. In particular, we can identify$\beta\in \mathbb{F}_q^*$ with automorphisms on $F_\nu$ that maps $\lambda\in F_\nu$ to $\beta\lambda$. Let $L_\nu$ be the fixed field of $\mathbb{F}_q^*$ in $F_\nu$, we have $[L_\nu:F]=q^\nu$ and $[F_\nu:L_\nu]=q-1$. Therefore, the place $(\frac{1}{T})$ is totally wildly ramified in $L_\nu/F$. We know from the definition of $L_\nu$ that $L_\nu\subset L_{\nu+1}$ and  $L_\infty$ is defined by the union $L_\infty=\cup_{\nu=1}^{\infty}L_\nu$.

\end{enumerate}

\begin{thm}\label{cft}
The three extensions $E/F$, $K_T/F$, and $L_\infty/F$ are linearly disjoint with each other, and their composite $E\cdot K_T\cdot L_\infty$ is the maximal abelian extension of $F$.
\end{thm}
\begin{proof}
See \cite{Hay74} Proposition 5.2 and Theorem 7.1.
\end{proof}

\section{The case $q=2$}
In this section, we set $A=\mathbb{F}_2[T]$ and $F=\mathbb{F}_2(T)$. Let $\phi_T=T+g_1\tau+g_2\tau^2$ be a Drinfeld $A$-module over $F$ of rank $2$ where $g_1\in F$ and $g_2\in F^*$. Suppose further that the mod $(T+i)$ Galois representations are surjective for $i\in\{0,1\}$. As $\bar{\rho}_{\phi, (T+i)}: G_F\rightarrow \GL_2(\F_2)$ factor through $\Gal(F(\phi[T+i])/F)$, we have $\Gal(F(\phi[T+i])/F)\cong \GL_2(\F_2)\cong S_3$, the symmetric group of order $6$. As the theory for mod $(T+1)$ representations is similar to mod $(T)$ representations, we focus on the mod $(T)$ representation tempararily. Since $S_3$ has a normal subgroup of index $2$, we have a degree $2$ subextension $M$ of $F(\phi[T])$ over $F$. The extension $M/F$ is abelian, so Theorem \ref{cft} implies $M$ must lie in either $E$, $K_T$, or $L_\infty$.  

\begin{enumerate}
\item[Case(1):]$M\subset E$, i.e. the extension $M/F$ is a constant extension.\\
We are able to characterize which type of $\phi_T=T+g_1\tau+g_2\tau^2$ will give a degree $2$ constant subextension $M$. 

\begin{lem}\label{q=2}
Given a rank $2$ Drinfeld module $\phi_T=T+g_1\tau+g_2\tau^2$ over $\F_2(T)$ and under the assumption that the mod $(T+i)$ representation is surjective, $i\in\{0,1\}$, the degree $2$ subextension $M$ of $F(\phi[T+i])$ over $F$ is a constant extension if and only if $g_1=0$.

\end{lem}

\begin{proof}
It's obvious that when $g_1=0$, the intermediate field $M$ is the constant extension $F(\zeta_3)$ over $F$ obtained by joining the $3$rd root of unity.

On the other hand, from the assumption that $F(\phi[T])/F\cong S_3$ and $\phi_T(x)=Tx+g_1x^2+g_2x^4=x(T+g_1x+g_2x^3)$, we can deduce $T+g_1x+g_2x^3$ is irreducible over the intermediate field $M$. In other words, there is an element $y\in F(\phi[T])$ whose minimal polynomial over $M$ is $T+g_1x+g_2x^3$ and $F(\phi[T])=M(y)$. As the minimal polynomial of $y$ over $M$ has all coefficients lie in $F$, we may consider the cubic extension $F(y)/F$. By \cite{MaWa17} Corollary 1.4(a), there is a primitive element $z$ for $F(y)/F$ with characteristic polynomial of the form $x^3+x+b$, where $b=\frac{T^4g_2^2}{g_1^6}\in F^*$. Furthermore, in the \cite{MaWa17} page 8, they characterized the unique degree $2$ subextension $M$ of $F(\phi[T])$. We thus have $M/F$ is the Artin-Schrier extension with polynomial $x^2+x=\frac{1+b^2}{b^2}$. Because $b\neq 0$, the Artin-schrier extension is not a constant extension. Thus we have proved that $g_1\neq 0$ would imply $M/F$ is not a constant extension.

\end{proof}

Since this case happens only when the given Drinfeld module is of the form 
$$\phi_T=T+g_2\tau^2,$$
we can just compare the two extensions $F(\phi[T])/F$ and $F(\phi[T+1]/F)$. They both have the constant extension $F(\zeta_3)$ as their degree $2$ subextension over $F$. Hence the mod $(T)$ and mod $(T+1)$ Galois representation have entanglement. i.e. $F(\phi[T])\cap F(\phi[T+1])\supsetneq F$. Thus we can never achieve adelic surjectivity in this case.

\item[Case(2):]$M\subset K_T$, i.e. $M$ is contained in $F(C[\mathfrak{a}])$ for some ideal $\mathfrak{a}$ of $A$.\\
From the given information, we can have the following lattice of field extensions:

\begin{tikzpicture}[node distance=1cm]
\node(L) {$F(\phi[T])$};
\node(X) [below of=L]{};
\node(C) [right of=X]{$F(C[\mathfrak{a}])$};
\node(M) [below of=X] {$M$};
\node(F) [below of=M]{$F$};

\draw (L)--(M);
\draw (M)--(F);
\draw (C)--(M);
\end{tikzpicture}

Since $M/F$ is a degree $2$ extension, we know the polynomial that generates $\mathfrak{a}$ can not be of degree $1$. Thus $\mathfrak{a}$ is not equal to $(T)$. Now we consider the extension $K=F(\sqrt[3]{g_2})$ of $F$ and its composite with $F(\phi[T])$, we have the following lattice:

\begin{tikzpicture}[node distance=1cm]
\node(L) {$K(\phi[T])$};
\node(X) [below of=L]{};
\node(C) [right of=X]{$K(C[\mathfrak{a}])$};
\node(M) [below of=X] {$MK$};
\node(F) [below of=M]{$K$};

\draw (L)--(M);
\draw (M)--(F);
\draw (C)--(M);
\end{tikzpicture}

Now we can consider the rank $2$ Drinfeld module over $K$ defined by 
$$\psi_T=T+g_1'+\tau^2,\ \ \text{$g_1'=\frac{g_1}{\sqrt[3]{g_2}}$}.$$

From the theory of Weil pairing for Drinfeld modules(\cite{Hei03}, Proposition 4.7.1), we know $\det\circ\bar{\rho}_{\psi,\mathfrak{a}}=\bar{\rho}_{C,\mathfrak{a}}: \Gal(K^{{\rm{sep}}}/K)\rightarrow (A/\mathfrak{a})^*$. Hence we get an extension $K(\psi[\mathfrak{a}])$ from $K(C[\mathfrak{a}])$. Moreover, $\psi$ and $\phi$ are isomorphic as Drinfeld modules over $K$. Thus we have $K(\psi[\mathfrak{a}])\cong K(\phi[\mathfrak{a}])$. Combining this fact with the lattice over $K$ above, we get
$$K(\phi[T])\cap K(\phi[\mathfrak{a}])\supset MK\supsetneq K.$$
The last containment is proper because $M/F$ is a degree $2$ extension while $K/F$ is of degree $1$ or $3$.
With this containment relationship in hand, we want to prove the following:

\begin{lem}\label{4.2}
 $F(\phi[T]) \cap F(\phi[\mathfrak{a}])\supset M \supsetneq F$. 
\end{lem}
\begin{proof}
It's enough to prove $F(\phi[\mathfrak{a}])\supset M$. If the degree of the extension $K/F$ is equal to $1$, then it's clear that $F(\phi[\mathfrak{a}])\supset M$. Hence we assume $K/F$ is a degree $3$ extension. Let's consider the following lattice of field extensions:

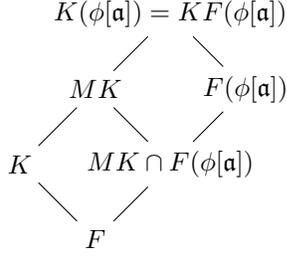
\begin{tikzpicture}[node distance=1cm]
\node(L) {$K(\phi[\mathfrak{a}])=KF(\phi[\mathfrak{a}])$};
\node(X) [below of=L]{};
\node(MK) [left of=X]{$MK$};
\node(FA) [right of=X]{$F(\phi[\mathfrak{a}])$};
\node(X1) [below of=MK]{};
\node(K) [left of=X1] {$K$};
\node(MKF) [right of=X1]{$MK\cap F(\phi[\mathfrak{a}])$};
\node(F)[below of=X1]{$F$};

\draw (L)--(MK);
\draw (L)--(FA);
\draw (MK)--(MKF);
\draw(FA)--(MKF);
\draw (MK)--(K);
\draw (MKF)--(F);
\draw(K)--(F);
\end{tikzpicture}

Firstly, we have $[K(\phi[\mathfrak{a}]):F(\phi[\mathfrak{a}])]=3\ \text{or}\ 1$. If the index is equal to $1$, then $$F(\phi[\mathfrak{a}])=K(\phi[\mathfrak{a}])\supset MK\supset M.$$
If the index is equal to $3$, then we apply the theorem of natural irrationalities to get
$$[MK:MK\cap F(\phi[\mathfrak{a}])]=[K(\phi[\mathfrak{a}]):F(\phi[\mathfrak{a}])]=3.$$
This implies $[MK\cap F(\phi[\mathfrak{a}]):F]=2$ because $[MK:F]=6$. This forces 
$MK\cap F(\phi[\mathfrak{a}])=M$, otherwise the compositum of $M$ and $MK\cap F(\phi[\mathfrak{a}])$ becomes a degree $4$ subextension of $MK$ over $F$. Thus $F(\phi[\mathfrak{a}])\supset M$.

\end{proof}
Thus there is an entanglement between mod $(T)$ and mod $\mathfrak{a}$ Galois representations associated to the Drinfeld module $\phi$.  We summarize this as the following lemma:

\begin{lem}\label{q=2.1}
Given  a rank $2$ Drinfeld module $\phi_T=T+g_1\tau+g_2\tau^2$ over $\F_2(T)$ and under the assumption that the mod $(T)$ representations are surjective. If the degree $2$ subextension $M$ of $F(\phi[T])$ over $F$ lies in some Carlitz extension $F(C[\mathfrak{a}])$, then $F(\phi[T]) \cap F(\phi[\mathfrak{a}])\supsetneq F$.

\end{lem}

The proof for mod $(T+1)$ Galois representation is essentially the same. Thus for Drinfeld modules $\phi_T=T+g_1\tau+g_2\tau^2$ that belong to this case, we will never attain adelic surjectivity.\\

\item[Case(3):]$M\subset L_\infty$, i.e. $M/F$ is wildly ramified at infinity.

If the subextension $M$ of $F(\phi[T])$ belongs to this case, then from \cite{MaWa17} page 8, we can characterize the subextension $M$ of $F(\phi[T])$. We have $M/F$ is the Artin-Schrier extension with polynomial $x^2+x=\frac{1+b^2}{b^2}$, where $b=\frac{T^4g_2^2}{g_1^6}$ (from Lemma \ref{q=2}, we know that $M/F$ is not constant extension if and only if $g_1\neq0$, hence the fraction is well-defined). Let $\alpha$ and $\alpha+1$ be the roots of $x^2+x=\frac{1+b^2}{b^2}$ and let $v_\infty:F^*\rightarrow \mathbb{Z}$ be the valuation of $F$ at infinity defined by $v_\infty(\frac{1}{T})=1$. Now we try to compute the valuation $v_\infty(\alpha)$. By computing the valuation of the constant term of $x^2+x=\frac{1+b^2}{b^2}=(x+\alpha)(x+\alpha+1)$, we have
$$2\cdot v_\infty(\frac{1+b}{b})=v_\infty(\frac{1+b^2}{b^2})=v_\infty(\alpha)+v_\infty(\alpha+1)=v_\infty(\alpha)+{\rm{min}}\{v_\infty(\alpha), 0\}.$$
This implies $v_\infty(\alpha)\in\mathbb{Z}$ because the left hand side of the above equality is an even integer.

Since $M/F$ is totally wildly ramified at infinity of degree $2$ and $M$ can be written as $F(\alpha)$, we have $v_\infty(c_1\alpha+c_2)=\frac{1}{2}$ for some $c_1 \text{ and } c_2\in F$, which is impossible. Thus in our case the degree $2$ subextension $M$ of $F(\phi[T])$ over $F$ can not be wildly ramified at infinity.

\end{enumerate}

Combining the three cases together, we thus prove the following theorem:
\begin{thm}\label{q=2}
Let $A=\mathbb{F}_2[T]$ and $F=\mathbb{F}_2(T)$. Let $\phi_T=T+g_1\tau+g_2\tau^2$ be a Drinfeld $A$-module over $F$ of rank $2$ where $g_1\in F$ and $g_2\in F^*$. The adelic Galois representation
$${\rho}_{\phi}:G_F\longrightarrow \varprojlim_{\mathfrak{a}}{\rm{Aut}}(\phi[\mathfrak{a}])\cong {\rm{GL_2}}(\widehat{A})$$
is not surjective.

\end{thm}

\section{The case $q=3$}
In this section, we set $A=\mathbb{F}_3[T]$ and $F=\mathbb{F}_3(T)$. Let $\phi_T=T+g_1\tau+g_2\tau^2$ be a Drinfeld $A$-module over $F$ of rank $2$ where $g_1\in F$ and $g_2\in F^*$.

\subsection{Assume $-g_2\in F^*$ non-square}\ \\
In this subsection, we will prove that it's impossible to have adelic surjectivity for Drinfeld $A$-modules of this type.\\
At the beginning of this subsection, we assume further that the mod $(T+i)$ Galois representation $\bar{\rho}_{\phi,(T+i)}$ of $\phi$ is surjective for all $i\in\{0,\pm1\}$. Therefore, we are able to consider the following composition of surjective maps:
$$\alpha:G_F\xrightarrow{\bar{\rho}_{\phi,(T+i)}}\GL_2(\F_3)\xrightarrow{\det}\{\pm1\}.$$
Hence there is a degree $2$ subextension $M_i$ of $F(\phi[T+i])$ over $F$. By Theorem \ref{cft}, the subextension $M_i$ over $F$ can only be a constant extension or  contained in a Carlitz extension because $L_\infty$ is a union of field extensions of $3$-power degree. Moreover, $F(\sqrt{-1})$ is the only degree $2$ constant extension over $F$. If any two of $M_i$ are constant extensions over $F$, say $M_{i_1}$ and $M_{i_2}$ where $i_1,\ i_2\in\{0,\pm1\}$, then we have $F(\phi[T+i_1])\cap F(\phi[T+i_2])\supset F(\sqrt{-1})\supsetneq F$. Thus we have proved the following lemma:

\begin{lem}\label{q=3const}
If any two of $M_i$ are constant extensions over $F$, then the two mod $(T+i)$ Galois representations $\bar{\rho}_{\phi,(T+i)}$ will have entanglement. Hence the adelic Galois representation for $\phi$ is not surjective

\end{lem}

If there is at most one constant extension among $M_i$, then there must be two other intermediate fields $M_{i_1}$ and $M_{i_2}$ with $i_1, i_2\in \{0,\pm1\}$ contained in the Carlitz extension of $F$. In other words, there are ideals $\mathfrak{a}_{i_1}$ and $\mathfrak{a}_{i_2}$ of $A$ such that $M_{i_1}\subset F(C[\mathfrak{a}_{i_1}])$ and $M_{i_2}\subset F(C[\mathfrak{a}_{i_2}])$.

\begin{claim}
$\mathfrak{a}_{i_1}\neq(T+i_1)$ and $\mathfrak{a}_{i_2}\neq(T+i_2)$.
\end{claim}
\begin{proof}
This proof relies on the assumption that $-g_2\in F^*$ is a non-square element. We will prove the claim for $\mathfrak{a}_{i_1}$, the proof for $\mathfrak{a}_{i_2}$ is the same.
From our construction of the map $\alpha$, we can see that the kernel of the composition of maps
$$\alpha:G_F\xrightarrow{\bar{\rho}_{\phi,(T+i)}}\GL_2(\F_3)\xrightarrow{\det}\{\pm1\}$$
is exactly $\SL_2(\F_3)$. 
We have $\det\circ\bar{\rho}_{\phi,(T+i_1)}=\bar{\rho}_{\psi,(T+i_1)}$ from Proposition 4.7.1 of \cite{Hei03}, where $\psi_T=T-g_2\tau$ is a rank $1$ Drinfeld $A$-module. Hence $M_{i_1}=F(\psi[T+i_1])=F(\sqrt{\frac{T+i_1}{g_2}})$. As $-g_2\in F^*$ is a non-square element, the field $F(\sqrt{\frac{T+i_1}{g_2}})$ is not equal to $F(C[T+i_1])=F(\sqrt{-(T+i_1)})$. This shows $\mathfrak{a}_{i_1}\neq(T+i_1)$
\end{proof}
The next part is to prove the following lemma, which is similar to what happened in the case ``$q=2$''.
\begin{lem}\label{lem5.2}
If the degree $2$ subextension $M_{i_1}$ of $F(\phi[T+i_1])$ over $F$ lies in some Carlitz extension $F(C[\mathfrak{a}_{i_1}])$, then $K(\phi[T+i_1]) \cap K(\phi[\mathfrak{a}_{i_1}])\supset M_{i_1}K$ where $K=F(\sqrt[8]{-g_2})$. We have the same statement when $i_1$ is replaced by $i_2$.
\end{lem}
\begin{proof}

First of all, $K/F$ is a degree $8$ extension because $-g_2$ is a non-square element in $F^*$. We consider the Dinfeld $A$-module over $K$ defined by $\psi_T=T+g_1'\tau-\tau^2$ with $g_1'=\frac{g_1}{\sqrt[4]{-g_2}}$. Then $\psi$ and $\phi$ are isomorphic Drinfeld modules over $K$. Furthermore, we have $\det\circ\bar{\rho}_{\psi,\mathfrak{a}}=\bar{\rho}_{C,\mathfrak{a}}:G_K\rightarrow (A/\mathfrak{a})^*$ from Proposition 4.7.1 of \cite{Hei03} again. Hence we can draw the following lattice of field extensions:

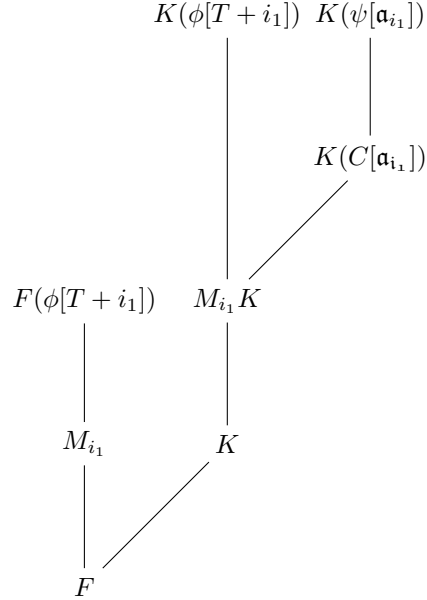
\begin{center}
\begin{tikzpicture}[node distance=1.9cm][segment length=1pt]
\node(X0) {};
\node(X1)[below of=X0]{};
\node(KL) [right of=X0]{$K(\phi[T+i_1])$};
\node(X2)[below of=KL]{};
\node(L) [below of=X1]{$F(\phi[T+i_1])$};

\node(KM)[below of=X2]{$M_{i_1}K$};

\node(KC)[right of=X2]{$K(C[\mathfrak{a_{i_1}}])$};
\node(KL')[above of=KC]{$K(\psi[\mathfrak{a}_{i_1}])$};

\node(M) [below of=L]{$M_{i_1}$};
\node(K) [below of=KM]{$K$};
\node(F) [below of=M]{$F$};

\draw (L)--(M);
\draw (M)--(F);
\draw (F)--(K);
\draw (K)--(KM);
\draw (KM)--(KL);
\draw(KM)--(KC);
\draw(KC)--(KL');
\end{tikzpicture}
\end{center}
As $\phi$ and $\psi$ are isomorphic Drinfeld modules over $K$, we have $K(\psi[\mathfrak{a}_{i_1}])=K(\phi[\mathfrak{a}_{i_1}])$. Hence we have proved $K(\phi[T+i_1])\cap K(\phi[\mathfrak{a}_{i_1}])\supset M_{i_1}K$.
\end{proof}

Now we have $K(\phi[T+i_1])\cap K(\phi[\mathfrak{a}_{i_1}])\supset M_{i_1}K$. If $M_{i_1}$ is not contained in $K$, then we get $$K(\phi[T+i_1])\cap K(\phi[\mathfrak{a}_{i_1}])\supset M_{i_1}K\supsetneq K.$$

Hence we are able to prove the following lemma: 
\begin{lem}\label{5.3.1}
$F(\phi[T+i_1])\cap F(\phi[\mathfrak{a}_{i_1}])\supset M_{i_1}\supsetneq F$
\end{lem}
\begin{proof}
Similar to lemma \ref{4.2}, it's enough to prove $F(\phi[\mathfrak{a_{i_1}}])\supset M_{i_1}$. We begin with the following:

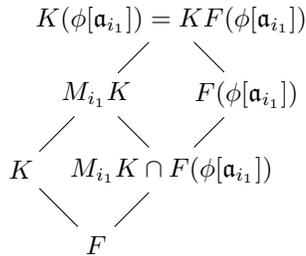
\begin{tikzpicture}[node distance=1cm]
\node(L) {$K(\phi[\mathfrak{a}_{i_1}])=KF(\phi[\mathfrak{a}_{i_1}])$};
\node(X) [below of=L]{};
\node(MK) [left of=X]{$M_{i_1}K$};
\node(FA) [right of=X]{$F(\phi[\mathfrak{a}_{i_1}])$};
\node(X1) [below of=MK]{};
\node(K) [left of=X1] {$K$};
\node(MKF) [right of=X1]{$M_{i_1}K\cap F(\phi[\mathfrak{a}_{i_1}])$};
\node(F)[below of=X1]{$F$};

\draw (L)--(MK);
\draw (L)--(FA);
\draw (MK)--(MKF);
\draw(FA)--(MKF);
\draw (MK)--(K);
\draw (MKF)--(F);
\draw(K)--(F);
\end{tikzpicture}

The index $[K(\phi[\mathfrak{a}_{i_1}]):F(\phi[\mathfrak{a}_{i_1}])]$ is either $1, 2, 4, \text{or}\ 8 $. If the index is equal to $1$, then
$$F(\phi[\mathfrak{a}_{i_1}])=K(\phi[\mathfrak{a}_{i_1}])\supset M_{i_1}K\supset M_{i_1}.$$

If $[K(\phi[\mathfrak{a}_{i_1}]):F(\phi[\mathfrak{a}_{i_1}])]=2$, then we have from natural irrationalities theorem that
$$[M_{i_1}K:M_{i_1}K\cap F(\phi[\mathfrak{a}_{i_1}])]=[K(\phi[\mathfrak{a}_{i_1}]):F(\phi[\mathfrak{a}_{i_1}])]=2.$$
Since $M_{i_1}$ and $K$ are linearly disjoint and $[M_{i_1}K\cap F(\phi[\mathfrak{a}_{i_1}]):F]=8$, we must have 
$$M_{i_1}\subset M_{i_1}K\cap F(\phi[\mathfrak{a}_{i_1}]).$$ 
Otherwise, $M_{i_1}K\cap F(\phi[\mathfrak{a}_{i_1}])=K$, which implies $  [K(\phi[\mathfrak{a}]):F(\phi[\mathfrak{a}_{i_1}])]=1$, a contradiction.

If $[K(\phi[\mathfrak{a}_{i_1}]):F(\phi[\mathfrak{a}_{i_1}])]=4$, then we have from natural irrationalities theorem that
$$[M_{i_1}K:M_{i_1}K\cap F(\phi[\mathfrak{a}_{i_1}])]=[K(\phi[\mathfrak{a}_{i_1}]):F(\phi[\mathfrak{a}_{i_1}])]=4.$$
Thus we have $[M_{i_1}K\cap F(\phi[\mathfrak{a}_{i_1}]):F]=4$. If $M_{i_1}$ is not contained in $M_{i_1}K\cap F(\phi[\mathfrak{a}_{i_1}])$, then we have $M_{i_1}K\cap F(\phi[\mathfrak{a}_{i_1}])\subset K$ is a degree $4$ extension over $F$. Hence $M_{i_1}K\cap F(\phi[\mathfrak{a}_{i_1}])=F(\sqrt[4]{-g_2})$. This implies the index $[K(\phi[\mathfrak{a}_{i_1}]):F(\phi[\mathfrak{a}_{i_1}])]$ is not equal to $4$, a contradiction. Thus 
$$M_{i_1}\subset M_{i_1}K\cap F(\phi[\mathfrak{a}_{i_1}]).$$

If $[K(\phi[\mathfrak{a}_{i_1}]):F(\phi[\mathfrak{a}_{i_1}])]=8$, then we have from natural irrationalities theorem that
$$[M_{i_1}K:M_{i_1}K\cap F(\phi[\mathfrak{a}_{i_1}])]=[K(\phi[\mathfrak{a}_{i_1}]):F(\phi[\mathfrak{a}_{i_1}])]=8.$$
Thus we have $[M_{i_1}K\cap F(\phi[\mathfrak{a}_{i_1}]):F]=2$. if $M_{i_1}$ is not contained in $M_{i_1}K\cap F(\phi[\mathfrak{a}_{i_1}])$, then we have $M_{i_1}K\cap F(\phi[\mathfrak{a}_{i_1}])\subset K$ is a degree $2$ extension over $F$. Hence $M_{i_1}K\cap F(\phi[\mathfrak{a}_{i_1}])=F(\sqrt{-g_2})$. This implies $\sqrt{-g_2}\in F(\phi[\mathfrak{a}])$ but $\sqrt[4]{-g_2}$ does not belong to $F(\phi[\mathfrak{a}_{i_1}])$. Hence the index $[K(\phi[\mathfrak{a}_{i_1}]):F(\phi[\mathfrak{a}])]$ is not equal to $2$, a contradiction. Thus
$$M_{i_1}\subset M_{i_1}K\cap F(\phi[\mathfrak{a}_{i_1}]).$$

Therefore, we can deduce to $M_{i_1}\subset F(\phi[\mathfrak{a}_{i_1}])$ in every case. This completes the proof.

\end{proof}
On the other hand, if $M_{i_1}$ is contained in $K=F(\sqrt[8]{-g_2})$, then $M_{i_1}$ is equal to the only degree $2$ subextension $F(\sqrt{-g_2})$ of $K$ over $F$. Now we consider the field extension $F(\phi[T+i_2])$ over $F$. From Lemma \ref{lem5.2}, we also have
$$K(\phi[T+i_2])\cap K(\phi[\mathfrak{a}_{i_2}])\supset M_{i_2}K.$$
If $M_{i_2}$ is contained in $K$ as well, then we get $M_{i_2}=F(\sqrt{-g_2})=M_{i_1}$. Thus we can deduce $$F(\phi[T+{i_1}])\cap F(\phi[T+i_2])\supsetneq F.$$
If $M_{i_2}$ is not contained in $K$, then we have, from the same argument as in Lemma \ref{5.3.1},
$$F(\phi[T+i_2])\cap F(\phi[\mathfrak{a}_{i_2}])\supsetneq F.$$
As a result, we have proved the following theorem:
\begin{thm}\label{5.3}
Let $\phi$ be a rank $2$ Drinfeld $A$-module over $F=\F_3(T)$ defined by 
$$\phi_T=T+g_1\tau+g_2\tau^2,\ \text{where $g_1\in F$ and $g_2\in F^*$ with $-g_2$ is a non-square element}.$$
Then the adelic Galois representation 
$${\rho}_{\phi}:G_F\longrightarrow \varprojlim_{\mathfrak{a}}{\rm{Aut}}(\phi[\mathfrak{a}])\cong {\rm{GL_2}}(\widehat{A})$$
is not surjective.
\end{thm}

\subsection{Assume $-g_2\in F^*$ a square element}
\ \\
In this subsection, we show that it's possible to attain adelic surjectivity for Drinfeld $A$-modules of this type. More specifically, we prove the Drinfeld module defined by $\varphi_T=T+\tau-T^2\tau^2$ can achieve adelic surjectivity.

From Lemma 6.3 of \cite{Zy11}, it will be enough to prove that the $\mathfrak{l}$-adic Galois representations for the Drinfeld module $\varphi$ is surjective for all prime ideals $\mathfrak{l}$ of $A$. For prime ideals $\mathfrak{l}$ with $\deg_T(\mathfrak{l})\geqslant 2$, the proof of $\mathfrak{l}$-adic surjectivity can be proved using the same procedure as in Zywina's work (\cite{Zy11} section 4-6) because $|\F_\mathfrak{l}|=|A/\mathfrak{l}|\geqslant 9$. Therefore, it will be enough to prove the $(T+i)$-adic Galois representations for $\varphi$ are surjective for $i\in\{0,\pm1\}$. 

We can prove the mod $(T+i)$ Galois representations $\bar{\rho}_{\varphi,(T+i)}$ for $\varphi$ is surjective by computing the order of Galois group $\Gal(F(\varphi[T+i])/F)$ directly. As $F(\varphi[T+i])/F$ is the splitting field of the polynomial $(T+i)x+x^3-T^2x^9$ over $F$, we use MAGMA to compute the order of $\Gal(F(\varphi[T+i])/F)$. It turns out that $|\Gal(F(\varphi[T+i])/F)|=48=|\GL_2(\F_3)|$. Thus $\bar{\rho}_{\varphi,(T+i)}$ is surjective for every $i\in\{0,\pm1\}$.

Now we want to prove that $(T+i)$-adic Galois representations $\rho_{\varphi,(T+i)}:G_F\rightarrow \GL_2(A_{(T+i)})$ is surjective for all $i\in\{0,\pm1\}$. Let $H_i$ be the image $\rho_{\varphi,(T+i)}(G_F)$. We know $H$ is a closed subgroup of $\GL_2(A_{(T+i)})$, and its determinant $\det(H)=\det\circ\rho_{\varphi,(T+i)}(G_F)=\det\circ\rho_{C,(T+i)}(G_F)=A_{(T+i)}^*$ because $-g_2\in F^*$ is a square element . Moreover, $H \mod (T+i)^2$ contains a non-scalar matrix that is congruent to the identity matrix modulo $(T+i)$ by Proposition 4.4 of \cite{Zy11}. Hence it will be enough if we can prove that Proposition 4.1 of \cite{PR09} is true for the case $n=2$ and $q=3$.

We first prove that Proposition 2.1 of \cite{PR09} is true for $n=2$ and $k=\F_3$.

\begin{lem}\label{5.5}
Let $H$ be an additive subgroup of the matrix ring $M_2(\F_3)$. Assume that $H$ is invariant under conjugation by $\GL_2(\F_3)$. Then either $H$ is contained in the group of scalar matrices or $H$ contains the group of matrices of trace $0$.
\end{lem}
\begin{proof}

Let $\left(\begin{array}{cc}a & b \\c & d\end{array}\right) \in H$ be an element. As $H$ is invariant under conjugation by $\GL_2(\F_3)$, we have
$$\left(\begin{array}{cc}-1 & 0 \\0 & 1\end{array}\right)\left(\begin{array}{cc}a & b \\c & d\end{array}\right)\left(\begin{array}{cc}-1 & 0 \\0 & 1\end{array}\right)=\left(\begin{array}{cc}a & -b \\-c & d\end{array}\right)\in H.$$
Thus $\left(\begin{array}{cc}a & 0 \\0 & d\end{array}\right)$ and $\left(\begin{array}{cc}0 & b \\c & 0\end{array}\right)$ belong to $H$. We can decompose $H$ as
$$H=H_0\oplus H_1,$$
where $H_0$ is the subgroup of $H$ containing diagonal matrices, and $H_1$ is the subgroup of $H$ containing anti-diagonal matrices.

\begin{claim}
$H_1$ is either trivial or equal to the group of all anti-diagonal matrices $$W_1=\left\{\left(\begin{array}{cc}0 & b \\c & 0\end{array}\right) \mid\ b \text{ and } c\in\F_3\right\}.$$
\end{claim}

\begin{proof}[Proof of claim]
Suppose $H_1$ is nontrivial, then there is some element $\left(\begin{array}{cc}0 & b \\c & 0\end{array}\right)\in H_1$ with $b, c\in\F_3$ not both equal to zero. If one of $b$ or $c$ is equal to $0$, then we consider the element 
$$\left(\begin{array}{cc}0 & 1 \\1 & 0\end{array}\right)\left(\begin{array}{cc}0 & b \\c & 0\end{array}\right)\left(\begin{array}{cc}0 & 1 \\1 & 0\end{array}\right)=\left(\begin{array}{cc}0 & c \\b & 0\end{array}\right)\in H_1.$$
These two element will then generate the whole group of anti-diagonal matrices $W_1$.
Thus we have $H_1=W_1$ in the case where one of $b$ or $c$ is equal to $0$.

If both $b$ and $c$ are not equal to $0$, then we are into the following two cases:

\begin{enumerate}
\item[Case 1: $b=c$]
\ \\
We consider the element
$$\left(\begin{array}{cc}0 & 1 \\1 & 1\end{array}\right)\left(\begin{array}{cc}0 & b \\c & 0\end{array}\right)\left(\begin{array}{cc}-1 & 1 \\1 & 0\end{array}\right)=\left(\begin{array}{cc}-c & c \\b-c & c\end{array}\right)\in H.$$
This element implies $\left(\begin{array}{cc}0 & c \\b-c & 0\end{array}\right)\in H_1$. As $b=c$, we have
$$\left(\begin{array}{cc}0 & c \\0 & 0\end{array}\right)\in H_1.$$ 
Thus we are back into the case where one of $b$ or $c$ is equal to $0$. Therefore, we can deduce that $H_1=W_1$

\item[Case 2: $b=-c$]
\ \\
Similar to case 1, we have $\left(\begin{array}{cc}0 & c \\b-c & 0\end{array}\right)\in H_1$. As we know $\left(\begin{array}{cc}0 & b \\c & 0\end{array}\right)\in H_1$ as well, we have 
$\left(\begin{array}{cc}0 & b+c \\b & 0\end{array}\right)\in H_1.$ Since $b=-c$, we have the element 
$$\left(\begin{array}{cc}0 & 0 \\b & 0\end{array}\right)\in H_1.$$
We are back to the case where one of $b$ or $c$ is zero again. Hence we have proved $H_1=W_1$.

\end{enumerate}
\end{proof}

Suppose now that $H_1$ is trivial, then $H$ is contained in the group of diagonal matrices. Let $\left(\begin{array}{cc}a & 0 \\0 & d\end{array}\right)\in H$, we consider the element

$$\left(\begin{array}{cc}1 & 1 \\0 & 1\end{array}\right)\left(\begin{array}{cc}a & 0 \\0 & d\end{array}\right)\left(\begin{array}{cc}1 & -1 \\0 & 1\end{array}\right)=\left(\begin{array}{cc}a & d-a \\0 & d\end{array}\right)\in H.$$

Since $H$ is contained in the group of diagonal matrices, we must have $d=a$. Therefore, $H$ is contained in the group of scalar matrices.

On the other hand, suppose $H$ contains the group $W_1$ of all anti-diagonal matrices. Then we consider the trace form $M_2(\F_3)\times M_2(\F_3)\rightarrow \F_3$ defined by $(A,B)\mapsto {\rm{tr}}(AB)$. This is a perfect pairing invariant under $\GL_2(\F_3)$. The orthogonal complement $H^{\perp}$ of $H$ is also invariant under conjugation action of $\GL_2(\F_3)$. Since $H$ contains $\left(\begin{array}{cc}0 & 1 \\0 & 0\end{array}\right)$ and $\left(\begin{array}{cc}0 & 0 \\1 & 0\end{array}\right)$, this forces $H^{\perp}$ to be contained in the group of scalar matrices. By taking orthogonal complement again, we can prove $H$ contains the group of matrices of trace $0$.

\end{proof}

Now we are able to prove Proposition 4.1 of \cite{PR09} is true for $n=2$ and $q=3$. Before stating the proposition, we introduce some notation.

Let $\mathfrak{p}=(T+i)$ be an ideal of $A$ where $i\in\{0, \pm1\}$. Let $\pi$ be a uniformizer of $A$ at $\mathfrak{p}$. The congruence filtration of $\GL_2(A_\mathfrak{p})$ is defined by
$$\begin{array}{cl}
G_\mathfrak{p}^0&:=\GL_2(A_\mathfrak{p})\\
G_\mathfrak{p}^i&:=1+\pi^iM_2(A_\mathfrak{p})\ \text{for all $i\geqslant 1$}

\end{array}.
$$
Its successive subquotients yield the following isomorphisms:

$$
\begin{array}{cl}
v_0:G_\mathfrak{p}^{[0]}:=&G_\mathfrak{p}^0/G_\mathfrak{p}^1\xrightarrow{\sim} \GL_2(\F_3)\\

v_i:G_\mathfrak{p}^{[i]}:=&G_\mathfrak{p}^i/G_\mathfrak{p}^{i+1}\xrightarrow{\sim} M_2(\F_3)\ \text{for all $i\geqslant 1$}\\
&[1+\pi^iy]\ \mapsto\  [y]

\end{array}.
$$
For any subgroup $H$ of $\GL_2(A_\mathfrak{p})$, we set $H^i:=H\cap G_\mathfrak{p}^i$ and $H^{[i]}:=H^i/H^{i+1}$.

\begin{prop}\label{5.6}
Let $H$ be a closed subgroup of $\GL_2(A_{\mathfrak{p}})$. Suppose that $\det(H)=A_{\mathfrak{p}}^*$, $H^{[0]}=\GL_2(\F_3)$, and $H^{[1]}$ contains a non-scalar matrix. Then we have $H=\GL_2(A_{\mathfrak{p}})$. 

\end{prop}
\begin{proof}
The only part in the proof of Proposition 4.1 of \cite{PR09} required ``$q\geqslant 4$'' is in the proof of $H^{[1]}=M_r(\F_\mathfrak{p})$. Thus we prove in our case that $H^{[1]}=M_2(\F_3)$, the rest will follow from the same argument as in the proof of \cite{PR09} Proposition 4.1.

Firstly, we consider the conjugation map
$$H^{[0]}\times H^{[1]}\rightarrow H^{[1]},\ \ ([g],[h])\mapsto [ghg^{-1}].$$
Since $H^{[0]}=\GL_2(\F_3)$ and $H^{[1]}\subset M_2(\F_3)$, the conjugation map implies $H^{[1]}$ is invariant under conjugation by $\GL_2(\F_3)$. As $H^{[1]}$ contains a non-scalar matrix, we know that $H^{[1]}$ contains the group $\mathfrak{sl}_2(\F_3)$ of matrices of trace $0$ by Lemma \ref{5.5}.
Now we consider the following commutative diagram with exact rows:
$$
\begin{array}{ccccclc}
0\longrightarrow &H^1/H^2&\rightarrow &H/H^2&\rightarrow&\GL_2(\F_3)&\rightarrow 0\\
                    &\downarrow{\det}&  &\downarrow{\det}&&\downarrow{\det}& \\
0\longrightarrow &(1+\pi A_\mathfrak{p}/(\pi)^2)^*&\rightarrow &(A_\mathfrak{p}/\mathfrak{p}^2)^*&\rightarrow&\F_3^*&\rightarrow 0
\end{array}.
$$
The right vertical map is certainly surjective, and the middle vertical map is surjective by our assumption that $\det(H)=A_\mathfrak{p}^*$. For the left vertical map, it is either trivial or surjective because $|(1+\pi A_\mathfrak{p}/(\pi)^2)^*|=3$. We claim that the left vertical map is also surjective. 

Suppose that the left vertical map is trivial, then any element in $H/H^2$ can be written as 
$$\left(\begin{array}{cc}a & b \\c & d\end{array}\right)\left(\begin{array}{cc}1+\pi\alpha & \pi\beta \\ \pi\gamma & 1+\pi\delta\end{array}\right)\ \text{where $\left(\begin{array}{cc}a & b \\c & d\end{array}\right)\in GL_2(\F_3)$ and $\left(\begin{array}{cc}1+\pi\alpha & \pi\beta \\ \pi\gamma & 1+\pi\delta\end{array}\right)\in H^1/H^2$}.$$
Now we compute the determinant 
$$\det\left(\begin{array}{cc}a & b \\c & d\end{array}\right)\det\left(\begin{array}{cc}1+\pi\alpha & \pi\beta \\ \pi\gamma & 1+\pi\delta\end{array}\right)=\det\left(\begin{array}{cc}a & b \\c & d\end{array}\right)\in \F_3^*.$$
This implies the middle vertical map is not surjective, which contradicts to our assumption. Thus the left vertical map must be surjective.

The surjectivity of the left vertical map then implies the composition $H^{[1]}\hookrightarrow M_2(\F_3)\xrightarrow{{\rm{tr}}}\F_3$ is surjective. Hence we have the following exact sequence of additive groups
$$0\longrightarrow\mathfrak{sl}_2(\F_3)\longrightarrow H^{[1]}\xrightarrow{\ \ {\rm{tr}}\ \ }\F_3\longrightarrow0.$$
This implies $H^{[1]}=M_2(\F_3)$ by the computing order of $H^{[1]}$.
\end{proof}

With Lemma \ref{5.5} and Lemma \ref{5.6} in hand, we can conclude this subsection by the following:

\begin{thm}\label{q=3-2}
Let $A=\F_3[T]$ and $F=\F_3(T)$. Let $\varphi$ be a rank $2$ Drinfeld $A$-module over $F$ defined by 
$$\varphi_T=T+\tau-T^2\tau^2.$$
Then the adelic Galois representation 
$${\rho}_{\varphi}:G_F\longrightarrow \varprojlim_{\mathfrak{a}}{\rm{Aut}}(\varphi[\mathfrak{a}])\cong {\rm{GL_2}}(\widehat{A})$$
is surjective.
\end{thm}

\section{The case $q=2^e\geqslant 4$}

In this section, $A=\F_q[T]$ and $F=\F_q(T)$ where $q\geqslant 4$ is some $2$-power. 
We will prove that it's possible to attain adelic surjectivity for Drinfeld $A$-modules over $F$ of rank $2$. More specifically, we show that Zywina's choice of Drinfeld module defined by $\varphi_T=T+\tau-T^{q-1}\tau^2$ can still achieve adelic surjectivity.

In \cite{Zy11}, Zywina proved the adelic Galois representation $\rho_{\varphi}:G_F\rightarrow \GL_2(\hat{A})$ attached to the Drinfeld $\F_q[T]$-module over $\F_q(T)$ defined by
$$\phi_T=T+\tau-T^{q-1}\tau^2$$ 
is surjective under the assumption that $q\geqslant 5$ is an odd prime power. However, he used the condition``$q\geqslant 5$'' only in proving Proposition 5.1 of \cite{Zy11}, and in Lemma 6.1 where he needed to apply Proposition 4.1 of \cite{PR09}. In the proof of Proposition 5.1 of \cite{Zy11}, it's enough to have $|\F_q^*|\geqslant 3$ because we only need to make sure there are at least $3$ elements in $\F_q[T]$ of degree $1$. In the statement of Proposition 4.1 of \cite{PR09}, it only requires $|\F_q|\geqslant 4$. Hence we can still prove Proposition 5.1 and Lemma 6.1 in \cite{Zy11} through the same procedure.
On the other hand, Zywina used the condition ``$q$ is odd'' in Lemma A.3 and in Lemma A.7 of \cite{Zy11}. In the proof of Lemma A.7, he applied Theorem 1 in \cite{Die80} where the theorem is for fields of odd characteristic. But this Theorem 1 is also true for finite field of even characteristic (see argument after Theorem 3 in \cite{Die80}), hence Lemma A.7 is still true in our case.
Therefore, it remains to prove that Lemma A.3 of \cite{Zy11} is true in our case.

\begin{lem}\label{6.1}
For each prime ideal $\mathfrak{p}$ of $A$, the group $\SL_2(A_\mathfrak{p})$ is equal to its commutator subgroup. The only normal subgroup of $\SL_2(A_\mathfrak{p})$ with simple quotient is the group $N=\{B\in\SL_2(A_\mathfrak{p})$,\ where $B\equiv I_2 \mod \mathfrak{p}\}$

\end{lem}
\begin{proof}
This proof is slightly modified from the proof of Lemma A.3 in \cite{Zy11} so that we can avoid using the condition ``$q$ is odd''.

First of all, we prove that $\SL_2(A_\mathfrak{p})$ is equal to its own commutator subgroup. Let $H$ be the commutator subgroup of $\SL_2(A_\mathfrak{p})$. It is a closed normal subgroup of $\SL_2(A_\mathfrak{p})$ and $\GL_2(A_\mathfrak{p})$. Now we define the followings

$$
\begin{array}{cll}
S^0&:=&\SL_2(A_\mathfrak{p})\\
S^i&:=&\{s\in\SL_2(A_\mathfrak{p}) \text{ with } s\equiv 1\mod \mathfrak{p}^i\}\ \text{for $i\geqslant 1$}
\end{array}
$$
For $i\geqslant 0$, we define $H^i=H\cap S^i$. We also define $S^{[i]}:=S^i/S^{i+1}$ and $H^{[i]}:=H^i/H^{i+1}$. There is a natural inclusion map from $H^{[i]}$ to $S^{[i]}$.

\begin{claim}
$H^{[i]}=S^{[i]}$ for all $i\geqslant 0$.
\end{claim}
\begin{proof}[Proof of claim]
Reduction modulo $\mathfrak{p}$ gives us that $S^{[0]}\cong \SL_2(\F_\mathfrak{p})$. Since $F_\mathfrak{p}$ is a finite field of even characteristic with order $\geqslant 4$, the group $\SL_2(F_\mathfrak{p})\cong {\rm{PSL}}_2(\F_\mathfrak{p})$ is a simple group. As $H^{[0]}$ is the commutator subgroup of $\SL_2(\F_\mathfrak{p})$, $H^{[0]}$ is a nontrivial normal subgroup of $\SL_2(\F_\mathfrak{p})$. Hence $H^{[0]}=S^{[0]}$.

For $i\geqslant 1$, let $\mathfrak{sl}_2(\F_\mathfrak{p})$ be the additive subgroup of $M_2(\F_\mathfrak{p})$ consists of all matrices of trace $0$. We have an isomorphism

$$S^{[i]}\xrightarrow{\sim} \mathfrak{sl}_2(\F_\mathfrak{p}),\ [1+\mathfrak{p}^iy]\mapsto [y].$$
Now we consider $H^{[i]}\hookrightarrow S^{[i]}\cong \mathfrak{sl}_2(\F_\mathfrak{p})$. As $H$ is a normal subgroup of $\GL_2(A_\mathfrak{p})$, $H^{[i]}$ is invariant under conjugation by $\GL_2(\F_\mathfrak{p})$. Thus Proposition 2.1 of \cite{PR09} implies either $H^{[i]}$ is contained in the group of scalar matrices or $H^{[i]}$ contains $\mathfrak{sl}_2(\F_\mathfrak{p})$. Thus it's enough to prove $H^{[i]}$ contains a non-scalar matrix.

Consider the commutator map $S^0\times S^i\rightarrow S^i$ defined by $(g,h)\mapsto ghg^{-1}h^{-1}$. This induces a map $S^{[0]}\times S^{[i]}\rightarrow S^{[i]}$ that takes values in $H^{[i]}$. Since $S^{[0]}\cong \SL_2(\F_\mathfrak{p})$ and $S^{[i]}\cong \mathfrak{sl}_2(\F_\mathfrak{p})$, the induced map becomes
$$\SL_2(\F_\mathfrak{p})\times \mathfrak{sl}_2(\F_\mathfrak{p})\rightarrow \mathfrak{sl}_2(\F_\mathfrak{p}),\ \ (s,X)\mapsto sXs^{-1}-X.$$
As the above map takes values in $H^{[i]}$, we can consider 
$$s=\left(\begin{array}{cc}\alpha & 0 \\0 & \alpha^{-1}\end{array}\right),\ X=\left(\begin{array}{cc}0 & 1 \\1 & 0\end{array}\right),\ \text{where $\alpha$ is a generator of $\F_\mathfrak{p}^*$}. $$
Then we compute
$$sXs^{-1}-X=\left(\begin{array}{cc}0 & \alpha^2-1 \\\alpha^{-2}-1 & 0\end{array}\right).$$
As $\alpha^2$ is not equal to $1$, the obtained matrix is not a scalar matrix. Hence $H^{[i]}$ contains a non-scalar matrix, so $H^{[i]}\supset \mathfrak{sl}_2(\F_\mathfrak{p})$. This proves $H^{[i]}=S^{[i]}$ for all $i\geqslant 1$.
\end{proof}
Since we have proved $H^{[i]}=S^{[i]}$ for all $i$, we can conclude that $H=\SL_2(A_\mathfrak{p})$.

Now let $N'$ be a normal subgroup of $\SL_2(A_\mathfrak{p})$ with simple quotient. As the Jordan-H\"older factors of $\SL_2(A_\mathfrak{p})$ are $\SL_2(\F_\mathfrak{p})$ and $\mathbb{Z}/2\mathbb{Z}$, we have $\SL_2(A_\mathfrak{p})/N'\cong \SL_2(\F_\mathfrak{p})\cong \SL_2(A_\mathfrak{p})/N$ because $\SL_2(A_\mathfrak{p})$ has no abelian quotients. This implies $N'=N$ because otherwise $NN'$ will be a proper normal subgroup of $\SL_2(A_\mathfrak{p})$ that contains $N$ and $N'$, a contradiction.
\end{proof}

Now we are able to prove the adelic surjectivity for the Drinfeld $A$-module $\varphi$ by going through the same procedure as in Zywina's work. We conclude this section by stating the theorem:

\begin{thm}\label{6.2}
Let $A=\F_q[T]$ and $F=\F_q(T)$ with $q=2^e\geqslant 4$. Let $\varphi$ be a rank $2$ Drinfeld $A$-module over $F$ defined by 
$$\varphi_T=T+\tau-T^{q-1}\tau^2.$$
Then the adelic Galois representation 
$${\rho}_{\varphi}:G_F\longrightarrow \varprojlim_{\mathfrak{a}}{\rm{Aut}}(\varphi[\mathfrak{a}])\cong {\rm{GL_2}}(\widehat{A})$$
is surjective.

\end{thm}

\bibliographystyle{alpha}
\bibliography{Exceptional_cases_on_adelic_surjectivity_for_Drinfeld_modules_of_rank_2_ver1}

\begin{thebibliography}{{Zyw}11}

\bibitem[Che20]{Chen20}
Chien-Hua Chen.
\newblock Surjectivity of the adelic galois representation associated to a
  drinfeld module of rank 3.
\newblock {\em Journal of Number Theory}, 2020.

\bibitem[Die80]{Die80}
Jean Dieudonn\'{e}.
\newblock {\em On the automorphisms of the classical groups}, volume~2 of {\em
  Memoirs of the American Mathematical Society}.
\newblock American Mathematical Society, Providence, R.I., 1980.
\newblock With a supplement by Loo Keng Hua [Luo Geng Hua], Reprint of the 1951
  original.

\bibitem[Hay74]{Hay74}
D.~R. Hayes.
\newblock Explicit class field theory for rational function fields.
\newblock {\em Trans. Amer. Math. Soc.}, 189:77--91, 1974.

\bibitem[MW17]{MaWa17}
Sophie Marques and Kenneth Ward.
\newblock A complete classification of cubic function fields over any finite
  field.
\newblock {\em arXiv e-prints}, page arXiv:1612.03534, 2017.

\bibitem[PR09]{PR09}
Richard Pink and Egon R\"{u}tsche.
\newblock Adelic openness for {D}rinfeld modules in generic characteristic.
\newblock {\em J. Number Theory}, 129(4):882--907, 2009.

\bibitem[Ser72]{Ser72}
Jean-Pierre Serre.
\newblock Propri\'{e}t\'{e}s galoisiennes des points d'ordre fini des courbes
  elliptiques.
\newblock {\em Invent. Math.}, 15(4):259--331, 1972.

\bibitem[{van}03]{Hei03}
Gerrit {van der Heiden}.
\newblock {\em Weil pairing and the Drinfeld modular curve}.
\newblock PhD thesis, University of Groningen, 2003.
\newblock Relation: $https://www.rug.nl/ date_submitted:2003$ Rights:
  University of Groningen.

\bibitem[{Zyw}11]{Zy11}
David {Zywina}.
\newblock {Drinfeld modules with maximal Galois action on their torsion
  points}.
\newblock {\em arXiv e-prints}, page arXiv:1110.4365, October 2011.

\end{thebibliography}
\end{document}